\UseRawInputEncoding
\documentclass{amsart} 
\usepackage{amssymb}
\usepackage{xypic}
\usepackage{mathpazo}
\usepackage{amssymb}
\usepackage{csquotes}
\usepackage{hyperref}

\newtheorem{theorem}{Theorem}[section]

\theoremstyle{definition}
\newtheorem{definition}[theorem]{Definition}
\newtheorem{example}[theorem]{Example}

\newtheorem{remark}[theorem]{Remark}

\numberwithin{equation}{section}
\newcommand{\beqa}{\begin{eqnarray*}}
	\newcommand{\eeqa}{\end{eqnarray*}}
\newcommand{\beqn}{\begin{eqnarray}}
	\newcommand{\eeqn}{\end{eqnarray}}

\newcommand{\ci}{\subseteq}
\newcommand{\pf}{\noindent {\bf Proof :} }

\renewcommand{\a}{\alpha}
\renewcommand{\b}{\beta}

\newcommand{\e}{\varepsilon}

\newcommand{\la}{\lambda}

\newcommand{\m}{\mu}

\newcounter{cnt1}
\newcounter{cnt2}
\newcounter{cnt3}
\newcommand{\blr}{\begin{list}{$($\roman{cnt1}$)$}
		{\usecounter{cnt1} \setlength{\topsep}{0pt}
			\setlength{\itemsep}{0pt}}}
	\newcommand{\bla}{\begin{list}{$($\alph{cnt2}$)$}
			{\usecounter{cnt2} \setlength{\topsep}{0pt}
				\setlength{\itemsep}{0pt}}}
		\newcommand{\bln}{\begin{list}{$($\arabic{cnt3}$)$}
				{\usecounter{cnt3} \setlength{\topsep}{0pt}
					\setlength{\itemsep}{0pt}}}
			\newcommand{\el}{\end{list}}
		
		\newtheorem{thm}{Theorem}
		
		\newtheorem{cor}[thm]{Corollary}

		\newtheorem{Def}[thm]{Definition}
		\newtheorem{prop}[thm]{Proposition}
		\newtheorem{rem}[thm]{Remark}
		\newcommand{\Rem}{\begin{rem} \rm}
			\newcommand{\bdfn}{\begin{Def} \rm}
				\newcommand{\edfn}{\end{Def}}

			\title{On the duality between Diameter Two \\ Properties and Ideals in Banach Spaces}
			
			\author{Sudeshna Basu}
			\address{Department of Mathematics and Statistics  \\Loyola University \\  MD 21210 \\ USA 
			}
			\email{sbasu1@loyola.edu, sudeshnamelody@gmail.com }
			\date{}
\begin{document}
				\maketitle
				
				\begin{abstract}
					We prove that the dual of an M ideal of a Banach space  inherits all the versions of  $w^*$ diameter two properties. 
					We give a counter example to show that the converse is not true. We use these results to explore these properties in  $C(K)$ spaces and its duals. 
				\end{abstract}
				\section{Introduction}
				Let $X$ be a {\it real} Banach space and $X^*$ its dual. We will denote by $B_X$, $S_X$ and $B_X(x, r)$ the closed unit ball, the unit sphere and the closed ball of radius $r >0$ and center $x$.

				\bdfn 
				Let $f \in X^*$, $\a > 0$ and $ C \subseteq X$ be a bounded subset of $X$.
				Then the set $S(C, f, \a) = \{x \in C : f(x) > Sup f(C) > \a \}$ is called the open slice of $C$, determined by $f$ and
				$\a.$ One can analogously define $w^*$-slices in $X^*.$  
			A set S  in a convex set $C \ci X$ is called a combination of slices
				of $K$, if for  there exist slices
				${S}_{i}$ of $C$, and scalars $\la_i, 0\leq a_i \leq 1 $ and a  {\em convex} combination such that   $S =
				\sum_{i=1}^{n} \la_i S_i$.
				
				\edfn

				\bdfn
				A Banach space $X$ has the slice diameter two property( Slice D2P), if every slice of the unit ball of $B_{X}$ has diameter 2,
				the diameter two property (D2P), if every nonempty relatively open subset of the unit ball of $B_{X}$ has diameter 2,and 
				 the strong diameter two property (Strong D2P), if every 
				convex combination of  slices  of the unit ball of $B_{X}$ has diameter 2.
				
				\edfn
				\begin{remark}
					Analogously, one can define $w^*$-Slice D2P, $w^*$-D2P and $w^*$-Strong D2P in a dual space by taking $w^*$-slices, $w^*$-open sets and $w^*$- small combination of slices respectively. 
				\end{remark}
				These properties are collectively known as diameter two properties. 
				
				
				The study of the diameter two  properties initiated in \cite{NW}, has gained momentum over the past two decades and continues to be a very active research topic.
				Many results in this direction have appeared over the years
				giving new geometrical properties in Banach spaces, entirely opposite to
				the well known geometric property, namely, the Radon Nikodym Property (RNP). For details, 
				see ,\cite{ALN,AGL,GL}--\cite{GLZ3},  \cite{HL}, \cite{K} -- \cite{Z}.
				In \cite{P}, it was proved that the diameter two  properties can be lifted from an M ideal to the whole Banach space. Also see \cite{HL}. A natural question that arose was,  what is the duality between M -Ideals and diameter two properties. In this work, we prove that the dual of an M ideal of a Banach space $X$ has $w^*$-strong D2P (respectively, $w^*$D-2P, $w^*$-Slice D2P) if $X^*$ has $w^*$-strong D2P (respectively, $w^*$D-2P, $w^*$-Slice D2P).
				We give a  counterexample to show the converse of these results may not be true. We use these results to explore these properties in  $C(K)$-spaces.

				\begin{definition}Let $Y\subseteq X$ be a subspace of $X.$ The annihilator of $Y$ in the dual space  $X^*$ is the subspace of $X^*$ defined by 
					$Y^\perp = \Big \{ x^*\in X^* : x^*(y) = 0\quad \forall y\in Y \Big \}$
				Let $Y$ be a closed subspace of $X.$ Then  $Y$ is said to be an  ideal in $X$ if $Y^{\perp}$ is the kernel of a norm one projection in $X^*$.
				If we look at $X$ as an embedding in $X^{**}$ via the map $J :X \rightarrow  X^{**}$ defined by
				$J(x)(x^{*}) = x^{*}(x)$ then $X$ is isometric to an ideal in $X^{**}$, see\cite{GKS}.
				\end{definition}
						
				
				\begin{definition}
					Let $X$ be a Banach space. A linear projection $P$ on $X$ is called 
				
						an $L$-projection if $\Vert x \Vert = \Vert Px \Vert + \Vert x-Px \Vert$ for all $x \in X.$
						and an $M$-projection if $\Vert x \Vert = Max \{\Vert Px \Vert, \Vert x-Px \Vert \}$ for all $x \in X.$
					Let $X$ be a Banach space. A closed subspace $Y \subset X$ is called an $M$-ideal if there exists an $L$-projection $P$ : $X^* \rightarrow X^*$ with Ker $P$ = $Y^\perp$.  
				\end{definition}
				
				We recall from Chapter I of \cite{HWW} that when $Y \subset X$ is an
				$M$-ideal, elements of $Y^\ast$ have unique norm-preserving
				extension to $X^\ast$ and one has the identification, $X^\ast =
				Y^\ast \oplus_1 Y^\perp$.
				

				\section{ Main Results  }
				
				\begin{prop} \label{Mideals strong D-2p}
					Let $M \subseteq X$ be an $M$-ideal. If $X^*$ has $w^*$-Strong D2P, then  $M^*$ has $w^*$-Strong D2P.
				\end{prop}
				\pf Suppose $X^*$ has $w^*$-Strong D2P .
				Suppose 
				$\e>0$ be arbitrary. 
				Let   $S = \sum_{i=1}^{n}\la_{i} S_{iM}$ be convex combination of slices of $B_{M^*}$, where $ S_{iM}= \{ m^* \in B_{M^*} / m^*(m_i)> 1 - \a_i \}.$
				Since $M$ is an $M$-
				ideal, for any $x^*\in X^*$ we have the unique decomposition, $x^*=
				m^*+ m^{\perp}$, where $m^* \in M^\ast$ and $m^\perp \in M^\perp$.
				Suppose we have $0<\m_i<\a_i.$ 
				Then 
				\beqa
				S_{iX} &= &\{ x^* \in B_{X^*} / x^*(m_i) > 1 - \mu_i\}\\
				&= & \{x^* \in B_{X^*}/ m^*(m_i) +m^{\perp}(m_i) > 1 - \mu_i\}\\
				&  \subseteq &S_{iM} \times \mu_{i} B_{M^\perp}\\
				\eeqa 
				Choose
				$\b_i = min ( \mu_i, \e).$ 
				Then \beqa
				S'_{iX}& =& \{ x^* \in B_{X^*} / x^*(m_i) > 1 - \b_i\} \subseteq S_{iM} \times \b_{i} B_{M^\perp}\\
				\Longrightarrow \sum_{i=1}^{n} \la_{i}S'_{iX}&  \subseteq &\sum_{i=1}^{n} \la_{i}S_{iM} \times \b_{i} B_{M^\perp}\\
				\eeqa
				Thus $2 = dia( \sum_{i=1}^{n} \la_{i}S'_{iX} )\leq  dia(\sum_{i=1}^{n} \la_{i}S_{iM} \times \b_{i} B_{M^\perp}) 
				\leq(\sum_{i=1}^{n} \la_{i}S_{iM}) +2\e $.
				
				This implies $ 2 - 2\e \leq  dia(\sum_{i=1}^{n} \la_{i}S_{iM}).$
				Since $\e$ is arbitrary, it follows that $ 2 = dia(\sum_{i=1}^{n} \la_{i}S_{iM}).$ 
				Also, since $ \|m_i\|=1 ,$ there exists $m^*_{i} \in B_{M^*}$ such
				that $m^*_{i}(m_{i}) > 1 - \b_i.$ Hence $m^*_{i} \in S'_{iX}.$
				Similarly, $$\sum_{i=1}^{n} \la_{i}m^{*}_{i} \in (\sum_{i=1}^{n} \la_{i}S'_{iX}) \Longrightarrow (\sum_{i=1}^{n} \la_{i}S'_{iX}) \neq  \emptyset$$  
				
				\qed
				
				Arguing similarly it follows that 
				\begin{cor} \label{M-idealslice D2P}
					Let $M \subseteq X$ be an $M$-ideal in $X$. If $X^*$ has the $w^*$-Slice D2P, then $M^*$ has $w^*$-Slice D2P .
				\end{cor}

				\begin{prop}\label{M-ideal D2P}
					Let  $M\subset X$ be  an $M$-ideal in $X$. If $X^*$ has $w^*$-D2P, then $M^*$ has $w^*$-D2P.
				\end{prop}
				\begin{proof}
					Let $\varepsilon>0.$ Let $V$ be any relatively $w^*$ open set of $B_{M^*}$. Choose $m_0^*$ in $S_{M^*}\cap V.$
					Then there exists $V_0 = \{ m^* \in B_{M^*} : \vert m^*(m_i) - m_0^*(m_i)\vert <\alpha_i \forall i=1,2,...n\}\subset V$ for some $n\in \mathbb{N}$ and $m_1,m_2,\ldots, m_n\in B_M.$
					Since $M$ is an $M$-ideal, $X^*= M^* \oplus_1 M^\perp.$
					For $m_0^*\in S_{M^*},$ there exists  $m_0\in B_M$ such that $\vert m_0^*(m_0)\vert > 1-\varepsilon.$
					Choose $\gamma>0$ such that $\vert m_0^*(m_0)\vert > 1-\varepsilon+\gamma.$
					Let $U_0=\{m^* \in B_{M^*} : \vert m_0^*(y_0) - m^*(y_0)\vert <\gamma \}.$
					Then, for $m^* \in U_0, \vert m^*(m_0)\vert > \vert m_0^*(m_0)\vert -\gamma > (1-\varepsilon + \gamma)-\gamma=1-\varepsilon.$
					Choose $0<\delta<\min\{\alpha_1 ,\alpha_2, \ldots \alpha_n, \gamma\}.$
					Let, $$W=\{x^*\in B_{X^*} : \vert x^*(m_i) - m_0^*(m_i)\vert<\delta , i=0,1,2,\ldots,n\}$$
					Clearly , $W$ is a relatively $w^*$-open subset of $B_{X^*}.$
					Also, $W\subset V_0 + \varepsilon B_{M^\perp}$
					Indeed , let $x^*\in W.$ Then there exists $m^*\in M^*$ and $m^\perp \in M^\perp$ such that $x^*=m^*+m^\perp.$\\
					Then $$\vert x^*(m_i)-m_0^*(m_i)\vert <\delta \quad \forall i=0,1,2,\ldots, n. $$
					$$\Rightarrow \vert (m^*+m^\perp)(m_i)-m_0^*(m_i)\vert <\delta \quad \forall i=0,1,2,\ldots,n. $$
					$$\Rightarrow \vert m^*(m_i)-m_0^*(m_i)\vert <\delta \hspace{1.8 cm} \forall i=0,1,2,\ldots,n. $$
					Hence, $\vert m^*(y_i)-m_0^*(y_i)\vert <\delta<\alpha \quad \forall i=0,1,2,\ldots,n $ \ and \ $\vert m^*(m_0)-m_0^*(m_0)\vert <\delta<\gamma$\\
					So , $m^*\in V_0$ and $m^*\in U_0$ , which implies $m^*\in V_0$ and $\Vert m^* \Vert > 1-\varepsilon .$\\
					Since, $\Vert x^* \Vert = \Vert m^* \Vert + \Vert m^\perp \Vert, $ it follows that $\Vert m^\perp \Vert < \varepsilon$ \quad i.e.\quad $m^\perp \in \varepsilon B_{Y^\perp}.$
					Thus $x^*=m^*+m^\perp \in V_0 + \varepsilon B_{M^\perp}.$
					Hence, $diam(W)\leqslant diam(V_0) + diam ( \varepsilon B_{M^\perp})$ which implies that, 
					$ 2 = diam (W) \leq diam(V) + diam ( \varepsilon B_{M^\perp})$.
					Hence $2 - 2\e \leq diam(V).$
					Since $\e$ is arbitrary, it follows that $M^*$ has $w^*$ D2P.
					
				\end{proof}
				The converse of Proposition \ref{Mideals strong D-2p}, Corollary \ref{M-idealslice D2P} and Proposition \ref{M-ideal D2P} are not true. We have the following example:
				\begin{example}
					Consider $Z=l_1\oplus_{\infty} \mathbb{R}$. Then $l_1$ is an $M$-ideal of $Z$. Also $Z^*= \l_{\infty} \oplus_{1} \mathbb{R}$. It is well known that  that $l_{\infty}$ has strong D-2P whereas, since $R$ has RNP, it has slices  of arbitrarily small diameter, see \cite{DU}. By Proposition 2.16. \cite{BS}, 
					$Z^*$ has $w^*$-slices  of arbitrarily small diameter and hence cannot have $w^*$-Slice-D2P. This implies  $Z^*$ also does not  have $w^*$D2P and $w^*$-Strong D2P. 
				\end{example}

				\begin{cor}
					Suppose $K$ is a compact Hausdorff space with an isolated point. If  
					$C(K,X)^{*}$ has the $w^*$-Strong D2P ( respectively $w^*$-D2P, $w^*$-Slice D2P ) then  $X^*$ has the $w^*$-Strong D2P
					( respectively $w^*$- D2P, $w^*$-Slice D2P )
					
				\end{cor}
				\pf For
				an isolated point $k_0 \in K$, the map $F \rightarrow \chi_{k_0}F$
				is a $M$-projection in $C(K,X)$ whose range is isometric to $X$. Hence $X^*$ is an M ideal in 
				$C(K,X)^*.$
				The rest follows from Proposition \ref{Mideals strong D-2p}, Corollary \ref{M-idealslice D2P}and Proposition \ref{M-ideal D2P} . \qed
				\begin{cor}
					Suppose $K$ is a compact Hausdorff space. If  $Y$ is a proper M ideal of $X$, then we have:
					\begin{enumerate}
						\item If $C(K,Y)$ has the Strong D2P (respectively D2P, Slice D2P) then $C(K,X)$ has the Strong D2P (respectively D2P, Slice D2P).
						
						\item If $C(K,X)^{*}$ has the $w^*$-Strong D2P (respectively $w^*$-D2P, $w^*$-Slice D2P) then  $C(K,Y)^{*}$ has the $w^*$-Strong D2P
						(respectively $w^*$- D2P, $w^*$-Slice D2P).
					\end{enumerate}
				\end{cor}
				\pf Since  $Y$ is an M-ideal of $X$, $C(K,Y)$ is an M ideal of $C(K,X)$ ( see \cite{HWW}).\\
				(i) follows from Propositions 3, 4 and 5 in \cite{HL}. Also see \cite{P}. \\(ii) follows from Proposition \ref{Mideals strong D-2p}, Corollary \ref{M-idealslice D2P} and Proposition \ref{M-ideal D2P}
				\qed
				\Rem
				It was proved in \cite{P}( also see \cite{HL}) that if $Y$ is an M Ideal of $X$, and $Y$ has diameter two  properties, then $X$ also has diameter two properties. The converse is not true.
				 Diameter two  properties are not inherited by M Ideals. 
				Indeed if $X=  R\oplus_{\infty} \ell_{\infty},$ then $X$ has Strong D-2P since $\ell_\infty$ has Strong D-2P (see \cite{HL}).
				$R$ is an M -ideal in $X$ and has RNP hence the unit ball has slices of arbitrarily small diameter. So $R$ does not have slice D-2P and hence does not have D-2P and strong D-2P.
			\end{rem}

		\thanks { {\bf Acknowledgment} 
		 This work was done  in Summer 2025 when the author  was visiting the National Institute of Science Education and Research Bhubaneshwar, India. She is grateful to Professor Anil Karn, Department of Mathematical Sciences for his support and warm hospitality. She thanks Dr. Susmita  Seal for fruitful discussions. 
				 The author is also grateful to Professor Bahram Roughani, Associate Dean College of Arts and Sciences, Loyola University for providing her with Dean's supplemental grant during her travel in Summer 2025.

			\end{document}